\documentclass[letterpaper,11pt]{amsart}

\usepackage[margin=1.2in]{geometry}
\usepackage{amsmath,amsthm,amssymb}
\usepackage{xspace,xcolor}
\usepackage[
  breaklinks,colorlinks,
  citecolor=teal,linkcolor=teal,urlcolor=teal
]{hyperref}
\usepackage[alphabetic]{amsrefs}
\usepackage[all]{xy}

\usepackage{tikz}

\theoremstyle{plain}
\newtheorem{thm}{Theorem}[section]
\newtheorem{lem}[thm]{Lemma}
\newtheorem{prop}[thm]{Proposition}

\newtheorem{cor}[thm]{Corollary}

\theoremstyle{definition}
\newtheorem{defi}[thm]{Definition}

\newtheorem{eg}[thm]{Example}

\theoremstyle{remark}
\newtheorem{rmk}[thm]{Remark}

\newtheorem*{ack}{Acknowledgments}

\numberwithin{equation}{section}

\def\N{{\mathbb N}}
\def\Z{{\mathbb Z}}
\def\Q{{\mathbb Q}}
\def\R{{\mathbb R}}
\def\C{{\mathbb C}}

\def\cI{\mathcal{I}}
\def\cJ{\mathcal{J}}

\def\cO{\mathcal{O}}

\def\I{\mathcal{I}}
\def\J{\mathcal{J}}
\def\O{\mathcal{O}}

\def\fa{\mathfrak{a}}

\def\a{\alpha}

\def\e{\eta}
\def\ep{\varepsilon}

\def\n{\nu}
\def\m{\mu}

\def\p{\pi}
\def\r{\rho}
\def\s{\sigma}

\def\D{\Delta}

\def\.{\cdot}
\def\^{\widehat}
\def\~{\widetilde}
\def\o{\circ}

\def\inj{\hookrightarrow}

\def\({\left(}
\def\){\right)}

\renewcommand{\and}{ \ \ \text{ and } \ \ }

\def\an{{\rm an}}

\def\reg{{\rm reg}}
\def\red{{\rm red}}

\def\num{{\rm num}}
\def\loc{{\rm loc}}

\DeclareMathOperator{\Pic} {Pic}

\DeclareMathOperator{\Sing} {Sing}

\DeclareMathOperator{\ord} {ord}
\DeclareMathOperator{\NS} {NS}

\DeclareMathOperator{\Cl} {Cl}
\DeclareMathOperator{\Cll} {Cl_\loc}
\DeclareMathOperator{\Cln} {Cl_\num}
\DeclareMathOperator{\NC} {NumCar}
\DeclareMathOperator{\Car} {Car}

\DeclareMathOperator{\Env} {Env}
\DeclareMathOperator{\Exc} {Exc}

\DeclareMathOperator{\Val} {Val}
\DeclareMathOperator{\DVal} {DivVal}

\newcommand{\ie}{{\rm i.e.\ }}

\begin{document}



\title{Valuation spaces and multiplier ideals on singular varieties}

\author{S\'ebastien Boucksom}
\address{CNRS - Institut de Math\'ematiques de Jussieu, 4 place Jussieu, 75252 Paris Cedex, France}
\email{boucksom@math.jussieu.fr}

\author{Tommaso de Fernex}
\address{Department of Mathematics, University of Utah, 
Salt Lake City, UT 48112-0090, USA}
\email{defernex@math.utah.edu}

\author{Charles Favre}
\address{CNRS - Centre de Math\'ematiques Laurent Schwartz, 
\'Ecole Polytechnique, 
91128 Palaiseau Cedex, France}
\email{favre@math.polytechnique.fr}

\author{Stefano Urbinati}
\address{Universit\`a degli Studi di Padova, Via
Trieste 63, 35121 Padova, Italy}
\email{urbinati.st@gmail.com}

\thanks    {The first author is partially supported by ANR projects MACK and POSITIVE. Second author was partially supported by NSF CAREER Grant
            DMS-0847059 and the Simons Foundation.
The third author is supported by the ERC-starting grant project "Nonarcomp" no.307856}

\thanks		{Compiled on \today. Filename \small\tt\jobname}


\begin{abstract} We generalize to all normal complex algebraic varieties the valuative characterization of multiplier ideals due to Boucksom-Favre-Jonsson in the smooth case. To that end, we extend the log discrepancy function to the space of all real valuations, and prove that it satisfies an adequate properness property, building upon previous work by Jonsson-Musta\c t\u a. We next give an alternative definition of the concept of numerically Cartier divisors previously introduced by the first three authors, and prove that numerically $\Q$-Cartier divisors coincide with $\Q$-Cartier divisors for rational singularities. These ideas naturally lead to the notion of numerically $\Q$-Gorenstein varieties, for which our valuative characterization of multiplier ideals takes a particularly simple form.
\end{abstract}

\maketitle

\tableofcontents

\section{Introduction}

Multiplier ideal sheaves are a fundamental tool both in complex algebraic and complex analytic geometry. They provide a way to approximate a 'singularity data', which can take the form of a (coherent) ideal sheaf, a graded sequence of ideal sheaves, a plurisubharmonic function, a nef $b$-divisor, etc..., by a coherent ideal sheaf satisfying a powerful cohomology vanishing theorem. For the sake of simplicity, we will focus on the case of ideals and graded sequences of ideals in the present paper. 

On a smooth (complex) algebraic variety $X$, the very definition of the multiplier ideal sheaf $\cJ(X,\fa^c)$ of an ideal sheaf $\fa\subset\cO_X$ with exponent $c>0$ is valuative in nature: a germ $f\in\cO_X$ belongs to $\cJ(X,\fa^c)$ iff it satisfies
$$
\n(f)>c\n(\fa)-A_X(\n)
$$
for all divisorial valuations $\n$, \ie all valuations of the form $\n=\ord_E$ (up to a multiplicative constant) with $E$ a prime divisor on a birational model $X'$, proper over $X$. Further, it is enough to test these conditions with $X'$ a fixed log resolution of $\fa$ (which shows that $\cJ(X,\fa^c)$ is coherent, as the direct image of a certain coherent fractional ideal sheaf on $X'$). Here we have set as usual $\n(\fa):=\min_{f\in\fa_x}\n(f)$ with $x=c_X(\n)$ the center of $\n$ in $X$, and
$$
A_X(\n):=1+\ord_E\left(K_{X'/X}\right)
$$
is the \emph{log discrepancy} (with respect to $X$) of the divisorial valuation $\n$. 

The multiplier ideal sheaf $\cJ(X,\fa_\bullet^c)$ of a graded sequence of ideal sheaves $\fa_\bullet=(\fa_m)_{m\in\N}$ is defined as the stationary value of $\cJ(X,\fa_m^{c/m})$ for $m$ large and divisible, but a direct valuative characterization was provided in \cite{FJ05b} in the $2$-dimensional case, in \cite{BFJ08} for all non-singular varieties, and in \cite{JM10,JM12} for the general case of regular excellent noetherian $\Q$-schemes. More specifically, for each divisorial valuation $\n$, subadditivity of $m\mapsto\nu(\fa_m)$ allows to define
$$
\nu(\fa_\bullet):=\lim_{m\to\infty}m^{-1}\nu(\fa_m)=\inf_{m\ge 1}m^{-1}\n(\fa_m)
$$
in $[0,+\infty)$. By \cite{BFJ08}, a germ $f\in\cO_X$ belongs to $\cJ(X,\fa_\bullet^c)$ iff there exists $0<\ep\ll 1$ such that
$$
\n(f)\ge(c+\ep)\n(\fa_\bullet)-A_X(\n)
$$
for all divisorial valuations $\n$.\footnote{In statements of this kind, if $f$ is a germ at $x\in X$, it is implicitely understood that we are only considering those $\n$ such that $x\in\overline{\{c_X(\n)\}}$, so that we can make sense of $\nu(f)$.} In other words, the latter condition is shown to imply the existence of $m\gg 1$ such that $\n(f)>c m^{-1}\n(\fa_m)-A_X(\n)$ for all divisorial valuations $\n$. 

\medskip

The definition of multiplier ideals was extended to the case of an arbitrary normal algebraic variety $X$ in \cite{dFH09}. If $\D$ is an effective $\Q$-Weil divisor on $X$ such that $K_X+\D$ is $\Q$-Cartier (\ie an \emph{effective $\Q$-boundary} in MMP terminology), the log discrepancy function $A_{(X,\D)}$ is a by now classical object (see for instance \cite{Kol97}). It allows to define the multiplier ideal sheaf $\cJ((X,\D);\fa^c)$ just as before for an ideal sheaf $\fa\subset\cO_X$, and then $\cJ\left((X,\D);\fa_\bullet^c\right)$ for a graded sequence of ideals $\fa_\bullet$ as the largest element in the family $\cJ((X,\D);\fa_m^{c/m})$. It is proven in \cite{dFH09} that there is a unique element in the family of ideals $\J((X,\D);\fa^c)$ with $\D$ ranging over all effective $\Q$-boundaries, which coincides with the multiplier ideal $\J(X,\fa^c)$ as defined in \cite{dFH09}. 

 Note in particular that $\cJ(X,\cO_X)=\cO_X$ iff there exists an effective $\Q$-boundary $\D$ such that the pair $(X,\D)$ is klt - which simply means that $X$ itself is log terminal when $X$ is already $\Q$-Gorenstein (\ie when $K_X$ is $\Q$-Cartier). 

In order to give a direct valuative description of these generalized multiplier ideals, one first needs to provide an adequate notion of log discrepancy for a divisorial valuation. As in \cite{dFH09}, this is done by setting for $\nu=\ord_E$ with $E$ a prime divisor on a birational model $X'$ proper over $X$
$$
A_X(\nu):=1+\ord_E(K_{X'})-\lim_{m\to\infty}m^{-1}\ord_E\cO_X\left(-mK_X\right),
$$
where $K_X$ is now an actual canonical Weil divisor on $X$ (as opposed to a linear equivalence class), $K_{X'}$ is the corresponding canonical Weil divisor on $X'$, and $\cO_X\left(-mK_X\right)$ is viewed as a fractional ideal sheaf on $X$. The definition is easily seen to be independent of the choices made. 

Our first main result is as follows. 
\begin{thm}
\label{t:1}
Let $\fa_\bullet$ be a graded sequence of ideal sheaves on a normal algebraic variety, and pick $c>0$. For every closed subscheme $N\subset X$ containing both $\Sing(X)$ and the zero locus of $\fa_1$ (and hence of $\fa_m$ for all $m$) and every $0<\ep\ll 1$, we have
$$
\J(X,\fa_\bullet^c)= 
\left\{f\in \cO_X \mid \n(f)\ge\nu(\fa_\bullet)-A_X(\n)+\ep\n(\I_N)\text{ for all divisorial valuations }\n\right\},
$$
with $\cI_N\subset\cO_X$ denoting the ideal sheaf defining $N$. 
\end{thm}
The key point in our approach is to construct an appropriate extension of $A_X$ to the space $\Val_X$ of all \emph{real} valuations on $X$, and to prove that it satisfies an adequate properness property, building upon \cite{JM10}. Once this is done, the last ingredient is a variant of Dini's lemma. The argument will also prove:

\begin{thm}\label{t:2}
Let $\fa_\bullet$ be a graded sequence of ideal sheaves on a normal algebraic variety, and pick $c>0$. Then 
$$
\J(X,\fa_\bullet^c)=\left\{f\in\cO_X\mid\n(f)>c\n(\fa_\bullet)-A_X(\n)\text{ for all real valuations }\n\right\}. 
$$
\end{thm}

\smallskip

In the special case $\fa_\bullet=\cO_X$, this last result show that there exists an effective $\Q$-boundary $\D$ with $(X,\D)$ klt iff $A_X(\n)>0$ for all real valuations $\n$. When $X$ admits an effective $\Q$-boundary $\D$ with $(X,\D)$ log canonical, one easily sees that $A_X\ge 0$. However, in that case the converse already fails in dimension three, as was recently shown by Yuchen Zhang for a normal isolated cone singularity~\cite{Zh13}.

\medskip

In the last part of the paper, we provide an alternative approach to the notion of \emph{numerically Cartier divisors} introduced in \cite{BdFF}. A Weil divisor on $X$ is said to be numerically ($\Q$-)Cartier if it is the push-forward of a $\pi$-numerically trivial ($\Q$-)divisor
for some (equivalently, any) resolution of singularities $\pi : X' \to X$. This naturally leads to the definition of a group of \emph{numerical divisor classes} $\Cln(X)$, defined as the quotient of the group of Weil divisors by numerically Cartier divisors. We prove that the abelian group $\Cln(X)$ is always finitely generated. The $\Q$-vector space $\Cln(X)_\Q$ is trivial when X is either $\Q$-factorial or has dimension $2$, thanks to Mumford's numerical pull-back. Building on an argument of Kawamata,
we further prove that every numerically $\Q$-Cartier divisor is already $\Q$-Cartier when $X$ has rational singularities. 

We say that $X$ is \emph{numerically $\Q$-Gorenstein} when $K_X$ is numerically $\Q$-Cartier. This means that for some (equivalently, any) resolution of singularities $\p:X'\to X$, $K_{X'}$ is $\p$-numerically equivalent to a $\p$-exceptional $\Q$-divisor, which is necessarily unique, and denoted by $K_{X'/X}^\num$. It relates to the log discrepancy function by
$$
A_X(\ord_E)=1+\ord_E\left(K_{X'/X}^\num\right)
$$
for all prime divisors $E\subset X'$. As a consequence of Theorem \ref{t:1}, we show:

\begin{thm}\label{t:3} Assume that $X$ is numerically $\Q$-Gorenstein, and let $\fa\subset\cO_X$ be an ideal sheaf. Let also $\p:X'\to X$ be a log resolution of $(X,\fa)$, so that $\p^{-1}\fa\cdot\cO_{X'}=\cO_{X'}(-D)$ with $D$ an effective Cartier divisor. For each exponent $c>0$ we then have
$$
\J(\fa^c) =\p_*\cO_{X'}\left(\lceil K_{X'/X}^\num-c D\rceil\right).
$$
\end{thm}

In dimension two, this result says that the multiplier ideals introduced in \cite{dFH09}
agree with the numerical multiplier ideals defined using Mumford's numerical pull-back.

Since the underlying variety of any klt pair has rational singularities, Theorem \ref{t:3} applied to $\fa=\cO_X$ yields: 

\begin{cor} Let $X$ be a normal algebraic variety. The following conditions are equivalent:
\begin{enumerate}
\item $X$ is $\Q$-Gorenstein and log terminal;
\item $X$ is numerically $\Q$-Gorenstein and $A_X(\ord_E)>0$ for all prime divisors $E$ on some (equivalently, any) log resolution $X'$ of $X$.
\end{enumerate} 
\end{cor}

\begin{ack} We are very grateful to Claire Voisin for providing a proof of Lemma \ref{lem:numsing}. We would also like to thank Mattias Jonsson for a key observation that helped us simplify the statements of the main results. 
\end{ack}

\section{Valuation spaces}
\label{s:val}
Throughout the paper we work over the field $\C$ of complex numbers. In this section we review some properties of valuation spaces, mostly following \cite{JM10}. 

\subsection{The space of real valuations}\label{sec:berko}
Let $X$ be an algebraic variety. By a \emph{valuation on $X$}, we will mean a real-valued valuation $\n$ on the function field of $X$ that is trivial on the base field and admits a center on $X$. Recall that the latter is characterized as the unique scheme point $c_X(\n)=\xi\in X$ such that $\n>0$ on the maximal ideal of $\cO_{X,\xi}$. We denote by $\Val_X$ the space of valuations on $X$, endowed with the topology of pointwise convergence. Mapping a valuation to its center defines an \emph{anticontinuous}\footnote{\ie the inverse image of an open subset is closed} map
$$
c_X:\Val_X\to X. 
$$
Every prime divisor $E$ \emph{over} $X$ (\ie in a normal birational model $X'$, proper over $X$) determines a valuation $\ord_E\in\Val_X$ given by the order of vanishing at the generic point of $E$. A \emph{divisorial valuation} is a valuation of the form 
$\nu=c\,\ord_E$ for some prime divisor $E$ over $X$ and some $c \in \R_+^*$. We denote by
$$
\DVal_X\subset\Val_X
$$
the set of divisorial valuations.

\subsection{Normalized valuation spaces}\label{sec:norm}

For every (coherent) ideal sheaf $\fa\subset\cO_X$ and every $\n\in \Val_X$, we set as usual
$$
\nu(\fa):=\min\left\{\n(f)\mid f\in\fa_{c_X(\n)}\right\}\in[0,+\infty).
$$

\begin{defi}\label{defi:regval} A \emph{normalizing subscheme} is a (non-trivial) closed subscheme of $X$ containing $\Sing(X)$. The \emph{normalized valuation space} defined by $N$ is 
$$
\Val_X^N:=\left\{\nu\in \Val_X\mid\nu(\I_N)=1\right\},
$$
with $\cI_N$ denoting the ideal sheaf defining $N$. 
\end{defi}
Note that 
$$
\R_+^*\cdot\Val_X^N=\left\{\n\in\Val_X\mid\n(\I_N)>0\right\}=c_X^{-1}(N), 
$$
which is thus open in $\Val_X$ and only depends on the Zariski closed set $N_\red$. We also clearly have
\begin{equation}\label{equ:cup}
\Val_X=\bigcup_{N\subset X}\R_+^*\cdot\Val_X^N,
\end{equation}
with $N$ ranging over all normalizing subschemes.

The point of introducing this terminology is that the normalized valuation space $\Val_X^N$ admits a simple description as a limit of simplicial complexes. 

\begin{defi} A \emph{good resolution} of a normalizing subscheme $N\subset X$ is a proper birational morphism $\pi:X_\pi\to X$ such that
\begin{itemize}
\item $X_\pi$ smooth;
\item $\pi$ is an isomorphism over $X\setminus N$;
\item $\pi^{-1}(N)\supset\Exc(\pi)$ both have pure codimension one, and $\pi^{-1}(N)_\red$ is a simple normal crossing divisor $\sum_{i\in I} E_i$ such that $E_J:=\bigcap_{j\in J}E_j$ is irreducible (or empty) for all $J\subset I$.
\end{itemize}
\end{defi}

Let $\p$ be a good resolution of $N$, and assume that $E_J$ as above is non-empty. At its generic point $\e_J$, the normal crossing condition guarantees that any choice of local equations $z_j\in\cO_{X_\pi,\e_J}$ for $E_j$, $j\in J$ yields a regular system of parameters. 
By Cohen's theorem we thus have $\^\O_{X_\p,\e_J} \cong \C[[z_j,\,j\in J]]$.
To every weight $w = (w_j)_{j \in J}\in\R_+^J$, we associate the monomial valuation $\n_w$ defined by
\begin{equation}\label{e1}
\nu_w \left(\sum_{\a\in\N^J} a_\a z^\a \right) := \min\left\{\sum_i w_i\a_i \mid a_\a \neq 0 \right\}. 
\end{equation}
Viewed as a valuation on $X$, $\nu_w$ is called a {\it quasi-monomial valuation} and belongs to $\Val_X^N$. The construction is independent of the choice of the local equations $z_j$, $j\in J$ (see~\cite[Sections 3--4]{JM10} for more details).

Note that $\n_w \in \Val_X^N$ if and only if $\sum_{j \in J} w_j\ord_{E_j}(N) = 1$. If we denote by $\D_\p^N \subset \Val_X^N$ the set of all normalized quasi-monomial valuations so obtained, then $\D_\p^N$ is 
a geometric realization of the \emph{dual complex} of $\sum_i E_i$, \ie the simplicial complex whose vertices are in bijection with 
$I$ and contains one simplicial face $\sigma_J$ joining all 
vertices $j \in J$ for any subset $J\subset I$ such that $E_J\neq \emptyset$.

Further, there is a natural continuous retraction 
$$
r_\p^N \colon \Val_X^N \to \D_\p^N,
$$ 
defined by letting $r_\pi^N(\nu)$ be the unique monomial valuation taking the value $\nu(E_j)$ on $E_j$. Note that $r_\p^N(\nu)$ belongs to the relative interior of the face $\s_J$, with 
$$
J:=\left\{j\in I\mid c_{X_\pi}(\nu)\in E_j\right\}. 
$$
If $\p'$ factors through $\p$ (in which case we write $\p' \ge \p$), then there is a natural inclusion $\D_\p^N \inj \D_{\p'}^N$. We then have: 

\begin{thm}[\protect{\cite{Ber,Thu}}]\label{thm:ber}
$$
\Val_X^N=\overline{\bigcup_\p\D_\p^N},
$$
where $\p$ runs over all good resolutions of $N$. More precisely, $\lim_\p r_\p^N(\n)=\n$ for each $\n\in\Val_X$. 
\end{thm}

A subset $\s \subset \Val_X^N$ is said to be a \emph{face} if $\s$ is a
face of $\D_\p^N$ for some $\p$. A face of $\D_\p^N$ is also called a \emph{$\p$-face}, and it can be endowed with a canonical affine structure induced from $\D_\p^N$.
A real valued function on $\Val_X^N$ is said to be \emph{affine} (resp., \emph{convex}) on a face if it is so in terms of the variable $w$ as in~\eqref{e1}.
We say that a property holds \emph{on small faces} if there exists $\p$
such that the property holds on
the faces of $\D_{\p'}$ for all $\pi'\ge\pi$.

\subsection{Functions defined by ideals}

\begin{prop}
\label{p:log a} Let $N\subset X$ be a normalizing subscheme and $\fa \subset \O_X$ be a coherent ideal sheaf. Then: 
\begin{enumerate}
\item
$\n\mapsto\nu(\fa)$ is a continuous function $\Val_X\to[0,+\infty)$, and concave on each face of $\Val_X^N$;
\item
for each good resolution $\p$ of $N$ we have $r_\p^N(\n)(\fa)\ge\n(\fa)$ on $\Val_X^N$. 
\end{enumerate}
If $N$ further contains the zero locus of $\fa$,  and $\p$ is a good resolution of $N$ dominating the blow-up of $\fa$, then $\n\mapsto\nu(\fa)$ is affine on the faces of $\D_\p^N$ and $r_\p^N(\n)(\fa)=\n(\fa)$ on $\Val_X^N$. In particular, $\nu\mapsto\n(\fa)$ is bounded on $\Val_X^N$. 
\end{prop}

\begin{proof} The proof is similar to \cite{BFJ08}, we briefly recall the argument. Using the same notation introduced in Section~\ref{s:val}, a valuation 
corresponding to a point in a 
face $\sigma$ of some $\D_\p^N$ is parametrized by 
$w = (w_j)_{j \in J}$ with $w_j \ge 0$ and $\sum_{j\in J} w_j\ord_{E_j}(N) = 1$. 
For every local function $f$ on $X$ we have
$\nu_w(f) = \min\left\{ \sum_iw_i\a_i \mid a_\a \neq 0 \right\}$ with
$f\o \p = \sum_\a a_\a z^\a$. Since $w \mapsto \nu_w(h)$ is the minimum of a collection of affine functions, it is concave. It follows that $\n(\fa)=\min\left\{\nu(f)\mid f\in\fa\right\}$ is convex on $\s$.
Moreover, if $N$ contains the zero locus of $\fa$ and $\pi$ dominates the blow-up of $\fa$, then this function is affine on $\s$. 
This proves~(a) and the first half of the last assertion. 

Still bearing in mind the above interpretation of valuations in terms of points on $\s$,
consider any valuation $\nu$ on 
$\C[z_1, \dots, z_n]$, and suppose $w_i = \nu(z_i)\ge 0$. Then $\nu(f) \ge \nu_w(f)$ 
for all $f\in \C[z_1,\dots, z_n]$, and equality holds when $f$ is a monomial. 
This yields the second half of the last assertion. 
\end{proof}

More generally, recall that a \emph{graded sequence of ideals} $\fa_\bullet = (\fa_m)_{m \ge 0}$ is a sequence of coherent ideal sheaves such that $\fa_m\.\fa_n \subset \fa_{m+n}$ for all $m,n$. This implies that
$\nu(\fa_m)$ is a subadditive sequence for each $\n\in\Val_X$, so that we can set
$$
\n(\fa_\bullet)=\lim_{m\to\infty}m^{-1}\nu(\fa_m)=\inf_{m\ge 1}m^{-1}\n(\fa_m).
$$

Proposition \ref{p:log a} generalizes to: 
\begin{prop}
\label{p:log a-bullet}
Let $N\subset X$ be a normalizing subscheme and let $\pi$ be a good resolution of $N$. For any graded sequence of ideal sheaves $\fa_\bullet = (\fa_m)_{m \ge 0}$, $\nu\mapsto\nu(\fa_\bullet)$ defines an upper semicontinuous function $\Val_X\to[0,+\infty)$ such that
\begin{enumerate}
\item $\n\mapsto\nu(\fa_\bullet)$ is concave and continuous on each face of $\Val_X^N$;
\item $r_\p^N(\n)(\fa_\bullet)\ge\n(\fa_\bullet)$ on $\Val_X^N$ for each good resolution $\p$ of $N$.   
\end{enumerate}
Furthermore, if $N$ contains the zero locus of $\fa_1$ (or, equivalently, of $\fa_m$ for all $m$), then $\nu\mapsto\nu(\fa_\bullet)$ is also bounded on $\Val_X^N$. 
\end{prop}
\begin{proof} Only the continuity on the faces is not a direct consequence of Proposition \ref{p:log a}. But since each face $\s$ of $\Val_X^N$ is a simplex, it follows from an elementary fact in convex analysis that $\n\mapsto\nu(\fa)$, being concave and lsc, is automatically continuous on $\s$. 
\end{proof}

\begin{rmk}\label{rem:Lip} As a consequence of~\cite[Theorem~B]{BFJ13}, one can show the following uniform Lipschitz property: assume that a given face $\s$ of $\Val_X^N$ has the property that the closure $Z\subset X$ of the center of some (equivalently, any) valuation of the relative interior of $\s$ is proper over $\C$. Then there exists $C>0$ such that for any graded sequence of ideal sheaves $\fa_\bullet$ the function $\n\mapsto\nu(\fa_\bullet)$ is Lipschitz continuous on $\s$ with Lipschitz constant $\le C\ord_Z(\fa_\bullet)$.
\end{rmk}

\section{The log discrepancy function}
\label{s:log discr}

Throughout this section, $X$ denotes a normal algebraic variety. 

\subsection{The log discrepancy of a divisorial valuation}
Let $K_X$ be a canonical Weil divisor on $X$, \ie the closure in $X$ of the divisor of a given rational form of top degree on $X_\reg$. The choice of $K_X$ induces on the one hand a graded sequence of fractional ideal sheaves $\left(\cO_X(-mK_X)\right)_{m\in\N}$, and on the other hand a canonical Weil divisor $K_{X'}$ for each birational model $X'$ of $X$.

Following \cite{dFH09,BdFF}, we define for all $m\ge 1$ the \emph{$m$-limiting log discrepancy function} as the unique homogeneous function $A_X^{(m)}:\DVal_X\to\R$ such that
$$
A_X^{(m)}(\ord_E)=1+\ord_E(K_{X'})-m^{-1}\ord_E\cO_X(-m K_X)
$$
for each prime divisor $E$ on a birational model $X'$. The definition is independent of the choices made, and the subadditivity of the sequence $\ord_E\cO_X(-mK_X)$ shows that $A_X^{(m)}$ converges pointwise to a function $A_X:\DVal_X\to\R$, the \emph{log discrepancy function}, with $A_X=\sup_{m\ge 1}A_X^{(m)}$. 

\subsection{The log discrepancy of a real valuation}

\begin{thm}\label{thm:extend} There is a unique way to extend $A_X$ and $A_X^{(m)}$ $(m\ge 1)$ to homogeneous, lower semicontinuous functions $\Val_X\to\R\cup\{+\infty\}$ such that the following properties hold for each normalizing subscheme $N\subset X$: 
\begin{itemize}
\item[(i)] $A_X^{(m)}$ and $A_X$ are convex and continuous on all faces of $\Val_X^N$, and $A_X^{(m)}$ is even affine on small faces;
\item[(ii)] on $\Val_X^N$ we have
$$
A_X^{(m)}=\sup_\pi A_X^{(m)}\circ r_\pi^N\quad\text{  and  }\quad A_X=\lim_\pi A_X\circ r_\pi^N,
$$ 
where $\pi$ runs over all good resolutions of $N$ and $r_\pi^N:\Val_X^N\to\D_\p^N$ is the corresponding retraction;
\item[(iii)] for each $a\in\R$, $\left\{A_X^{(1)}\le a\right\}\cap\Val_X^N$ is compact, and $A_X^{(m)}$ converges uniformly to $A_X$ on this set.  
\end{itemize}
Further, we have $A_X=\sup_{m\ge 1}A_X^{(m)}$ on $\Val_X$. 
\end{thm}

\begin{rmk}
Combined with Remark~\ref{rem:Lip}, our proof will show that $A_X$ is in fact Lipschitz continuous on any face of $\Val_X^N$ containing valuations with proper center in $X$.
\end{rmk}

\begin{rmk}
We do not know whether $A_X \ge A_X \o r_\p^N$ holds on $\Val_X^N$ for $\p$ large enough in general.
\end{rmk}
Theorem \ref{thm:extend} will be proved by reduction to the smooth case. The next result summarizes the required properties for $X$ smooth, all of which are contained in \cite{JM10}. 

\begin{lem}\label{lem:JM} Assume that $X$ is smooth. Then Theorem \ref{thm:extend} holds; further, if $\fa_\bullet$ is a graded sequence of ideal sheaves on $X$ and $N$ is a normalizing subscheme containing the zero locus of $\fa_1$ (and hence of $\fa_m$ for all $m$), $m^{-1}\n(\fa_m)\to\nu(\fa_\bullet)$ uniformly for $\n\in\{A_X\le a\}\cap\Val_X^N$, for each $a\in\R$. 
\end{lem}
\begin{proof} When $X$ is smooth, properties (i) and (ii) of Theorem \ref{thm:extend} follow from \cite[Proposition 5.1, Corollary 5.8]{JM10}. The compactness of $\{A_X\le a\}\cap\Val_X^N$ is a consequence of the Skoda-Izumi inequality, just as in the proof of \cite[Proposition 5.9]{JM10}. By Dini's lemma and the subadditivity of $(\n(\fa_m))_{m\in\N}$, the uniform convergence is equivalent to the continuity of $\nu\mapsto\nu(\fa_\bullet)$ on $\{A_X\le a\}\cap\Val_X^N$. This is proved exactly as in \cite[Section 6.1]{JM10} (itself building on \cite[Theorem 3.9]{BFJ08}), by exploiting the superadditivity property of the sequence of (classical) multiplier ideals $\cJ(X,\fa_\bullet^m)$ $(m\in\N)$ on the smooth variety $X$.  
\end{proof}

\begin{proof}[Proof of Theorem \ref{thm:extend}] Uniqueness is clear: since the rational points of each dual complex $\D_\p^N$ consist of divisorial valuations, $A_X$ and $A_X^{(m)}$ are uniquely determined on $\D_\p^N$ by (i), and hence on $\Val_X^N$ by (ii). By homogeneity, they are also uniquely determined on 
$$
\Val_X=\bigcup_N\R_+^*\cdot\Val_X^N.
$$
In order to prove existence, we fix the choice of a projective birational morphism $\mu:X'\to X$ such that $X'$ is smooth and $\mu$ is an isomorphism over $X_\reg$. 

We claim that there exists a $\mu$-exceptional effective divisor $D$ on $X'$ such that the graded sequence of fractional ideal sheaves
$$
\fa_m := \mu^{-1}\O_X(-mK_X)\cdot\O_{X'}(mK_{X'})\cdot\cO_{X'}(-mD)
$$
is a sequence of actual ideal sheaves. To see this, it is enough to choose $D$ such that $\fa_1\subset\cO_{X'}$. But we may add a Cartier divisor $Z$ to $K_X$ so that $K_X + Z$ is effective.
The divisorial part of the ideal sheaf
$$
\m^{-1 }\O_X (-K_X -Z)\cdot\cO_{X'}
$$ 
coincides with $\O_{X'}(- K_{X'} -\m^* Z)$ up to a $\m$-exceptional divisor $D$, and we get $\fa_1\subset\cO_{X'}$ as desired for this choice of $D$. 

Note that we have by definition
$$
A_X^{(m)}(\n)=A_{X'}(\n)+\n(D)-m^{-1}\n(\fa_m)
$$
and 
$$
A_X(\n)=A_{X'}(\n)+\n(D)-\nu(\fa_\bullet)
$$
for all $\n\in\DVal_X\simeq\DVal_{X'}$. Using the canonical homeomorphism $\Val_X\simeq\Val_{X'}$, we can now use these formulas to \emph{define} $A_X$ and $A_X^{(m)}$ on $\Val_X$. Propositions \ref{p:log a} and \ref{p:log a-bullet} already show that $A_X$ and $A_X^{(m)}$ are homogeneous and lsc on $\Val_X$.  It remains to see that they satisfy (i), (ii) and (iii) of Theorem \ref{thm:extend}. 

Let $N\subset X$ be a given normalizing subscheme. Each good resolution $\pi':X'_{\pi'}\to X'$ of $N':=\mu^{-1}(N)$ induces a good resolution $\pi:=\mu\circ\pi'$ of $N$ such that $\D_{\pi'}^{N'}=\D_\p^N$, and the retractions $r_\pi^N:\Val_X^N\to\D_\p^Ni^N$ and $r_{\pi'}^{N'}:\Val_{X'}^{N'}\to\D_{\pi'}^{N'}$ identify modulo the canonical homeomorphism $\Val_{X'}\simeq\Val_X$.

Since $N'$ contains the support of $D$ and the zero locus of $\fa_1$, Proposition \ref{p:log a} shows that $\nu\mapsto\nu(D)-\nu(\fa_\bullet)$ is bounded and lower semicontinuous on $\Val_{X'}^{N'}$, and continuous and convex on the faces of $\Val_{X'}^{N'}$, while $\nu\mapsto\n(D)-m^{-1}\n(\fa_m)$ is affine on small faces. It follows that $A_X$ and $A^{(m)}_{X}$ satisfy (i). 

Now pick $a\in\R$, and set for simplicity
$$
K:=\{A_X^{(1)}\le a\}\cap\Val_X^N.
$$
Since $\nu(D)-\nu(\fa_1)$ is bounded for $\n\in\Val_X^N$, $K$ is contained in $\{A_{X'}\le a'\}\cap\Val_{X'}^{N'}$ for some $a'\in\R$, and hence is compact by Lemma \ref{lem:JM} and the lower semicontinuity of $A_X^{(1)}$. Lemma \ref{lem:JM} also shows that $m^{-1}\nu(\fa_m)\to\nu(\fa_\bullet)$ uniformly for $\n\in K$, which proves (iii).  

Let $\nu\in \Val_X^N$. If $A_X(\n)$ is finite, then so is $A_{X'}(\n)$, and $r_{\pi'}^{N'}(\nu)$ stays in the compact set
$$
\{A_{X'}\le A_{X'}(\n)\}\cap\Val_{X'}^{N'}
$$ 
since $A_{X'}\ge A_{X'}\circ r_{\p'}^{N'}$. For each $m$ fixed we have by Proposition \ref{p:log a}
$$
\nu(\fa_m)=r_{\p'}^{N'}(\nu)(\fa_m)
$$
for $\p'$ large enough, which proves that $A^{(m)}_X$ satisfies (ii). By uniform convergence of $m^{-1}\nu'(\fa_m)$ to $\nu'(\fa_\bullet)$ for $\nu'\in\{A_{X'}\le A_{X'}(\n)\}\cap\Val_{X'}^{N'}$, we infer
$$
\nu(\fa_\bullet)=\lim_{\pi'}r_{\pi'}^{N'}(\nu)(\fa_\bullet), 
$$
so that $A_X$ satisfies (ii) on the locus of $\Val_X^N$ where it is finite. If now $\nu\in \Val_X^N$ has $A_{X}(\nu)=+\infty$, then 
$$
\lim_{\pi'}A_{X'}(r_{\pi'}^{N'}(\nu))=A_{X'}(\nu)=+\infty,
$$
while $r_{\pi'}^{N'}(\nu)(\fa_\bullet)$ remains bounded, and we thus get (ii) at $\nu$ as well. The same argument also proves the last assertion of Theorem \ref{thm:extend}. 
\end{proof}

\section{Valuative characterization of multiplier ideals}
\label{s:mult-ideals}

\subsection{Multiplier ideal sheaves}
We briefly recall the definition of multiplier ideals in the context of a normal variety, as introduced in~\cite{dFH09}. We follow the presentation of \cite[Section~3]{BdFF}, which is phrased in the language of $b$-divisors, and therefore closer to our present valuative point of view. Indeed, it suffices to recall that a $b$-divisor is nothing but a homogeneous function on $\DVal_X$, with the extra property that it is non-zero on only finitely many prime divisors of $X$. 

If $\fa\subset\cO_X$ is a coherent ideal sheaf and $c$ is a positive real number, the \emph{$m$-limiting multiplier ideal sheaf} of $\fa^c$ is defined as
$$
\J_m(X,\fa^c):=\left\{f\in\cO_X\mid\n(f)>c\n(\fa)-A_X^{(m)}(\n)\text{ for all }\n\in\DVal_X\right\},
$$
where $A_X^{(m)}$ is the $m$-limiting log discrepancy function from Section \ref{s:log discr}. This is a reformulation of \cite[Definition 3.7]{BdFF}, which is phrased in the equivalent language of $b$-divisors. It is proved in \emph{op.cit.} that $\cJ_m(X,\fa^c)$ is actually coherent. 

We have $\J_m(X,\fa^c) \subset \J_l(X,\fa^c)$ whenever $m$ divides $l$, and the \emph{multiplier ideal sheaf} $\J(X,\fa^c)$ can thus be defined as the unique maximal element of the family $(\J_m(X,\fa^c))_{m\in\N }$. By \cite[Theorem 3.8]{BdFF}, $\cJ(X,\fa^c)$ is also the largest element in the family of 'classical' multiplier ideals $\cJ((X,\D);\fa^c)$, where $\D$ runs over all effective $\Q$-Weil divisors on $X$ such that $K_X+\D$ is $\Q$-Cartier (so that $(X,\D)$ is a pair in the sense of Mori theory).

More generally, when $\fa_\bullet=(\fa_m)_{m\in\N}$ is a graded sequence of (coherent) ideal sheaves, the multiplier ideal $\J(X,\fa_\bullet^c)$ is defined as the maximal element of the family $\J(X,\fa_m^{c/m})$ \cite[Definition 3.12]{BdFF}. Equivalently, $\cJ(X,\fa_\bullet^c)$ is also the largest element in the family $\cJ_m\left(X,\fa_m^{c/m}\right)$, cf. \cite[Lemma 3.13]{BdFF}.

\subsection{Valuative characterization}

Theorems~\ref{t:1} and \ref{t:2} from the introduction are restated together in the following result. 
\begin{thm}\label{t:mult-ideal}
Let $X$ be a normal variety. Let $\fa_\bullet$ be a graded sequence of (coherent) ideal sheaves on $X$, 
and let $c > 0$ be a real number. Then the following two characterizations
of the multiplier ideal sheaf $\J(X,\fa_\bullet^c)\subset\cO_X$ hold:
\begin{enumerate}
\item
$$
\J(X,\fa_\bullet^c)= \left\{f\in\cO_X \mid\nu(f)>c\nu(\fa_\bullet)-A_X(\n)\text{ for all }\n\in\Val_X\right\}.
$$
\item
For every normalizing subscheme $N\subset X$ containing the zero locus of $\fa_1$ (and hence of $\fa_m$ for all $m$) and every $0<\ep\ll 1$ we have 
$$
\J(X,\fa_\bullet^c)=\left\{f\in\cO_X\mid\n(f)\ge c\n(\fa_\bullet)-A_X(\n)+\ep\n(\cI_N)\text{ for all }\n\in\DVal_X\right\}. 
$$
\end{enumerate}
\end{thm}

A key ingredient in the proof is the following simple variant of Dini's lemma. 
\begin{lem}\label{lem:dini} Let $Z$ be a Hausdorff topological space, and let $\phi_m:Z\to\R\cup\{+\infty\}$ be a non-decreasing sequence of lower semicontinuous functions converging to $\phi$. Assume also that each sublevel set $\left\{\phi_1\le a\right\}$ with $a\in\R$ is compact. Then $\inf_Z\phi_m\to\inf_Z\phi$.
\end{lem}
\begin{proof} Assume first that $\inf_Z\phi<+\infty$ (which is the only case we shall actually use), and let $\ep>0$. Setting
$$
K_m:=\left\{\phi_m\le\inf_Z\phi-\ep\right\}
$$ 
defines a decreasing sequence of compact sets, by lower semicontinuity of $\phi_m$ and the compactness of sublevel sets of $\phi_1$. Since
$$
\bigcap_{m\in\N}K_m=\left\{\phi\le\inf_Z\phi-\ep\right\}=\emptyset,
$$
it follows that $K_m=\emptyset$ for all $m\gg 1$, \ie $\inf_Z\phi_m>\inf_Z\phi-\ep$ for all $m\gg 1$. 
In case $\phi\equiv+\infty$, the same argument applies, fixing $A>0$ instead of $\ep$ and replacing $K_m$ with $K_m':=\left\{\phi_m\le A\right\}$.
\end{proof}

\begin{proof}[Proof of Theorem \ref{t:mult-ideal}]
Let $N\subset X$ be a normalizing subscheme containing the zero locus of $\fa_1$. Let also $U\subset X$ be an affine open set and pick $f\in\cO(U)$. The theorem will follow from the equivalence between the following properties:
\begin{itemize}
\item[(i)] $f\in\cJ(X,\fa_\bullet^c)(U)$.
\item[(ii)] $\n(f)>c m^{-1}\n(\fa_m)-A_X^{(m)}(\n)$ on $\DVal_U^N$ for all $m$ large and divisible.
\item[(iii)] $\n(f)>c m^{-1}\n(\fa_m)-A_X^{(m)}(\n)$ on $\Val_U^N$ for all $m$ large and divisible.
\item[(iv)] $\n(f)>c\n(\fa_\bullet)-A_X(\n)$ on $\Val_U^N$.
\item[(v)] $\n(f)\ge c\n(\fa_\bullet)-A_X(\n)+\ep$ on $\Val_U^N$ for some $0<\ep\ll 1$.
\item[(vi)] $\n(f)\ge c\n(\fa_\bullet)-A_X(\n)+\ep$ on $\DVal_U^N$ for some $0<\ep\ll 1$.
\item[(vii)] $\n(f)\ge c\n(\fa_\bullet)-A_X(\n)+\ep\n(\I_N)$ on $\DVal_U$ for some $0<\ep\ll 1$. 
\end{itemize}
Let us first check (i)$\Longleftrightarrow$(ii). Since $U$ is affine, $\cJ(X,\fa_\bullet^c)(U)$ is the largest element in the family of ideals $\J_m(X,\fa_m^{c/m})(U)$ of $\cO(U)$, and (i) thus amounts to $\n(f)>c m^{-1}\n(\fa_m)-A_X^{(m)}(\n)$ on $\DVal_U$ for all $m$ large and divisible, which implies (ii). Conversely, (ii) implies (i) since for any $\nu\in\DVal_U$ centered outside $N\supset\Sing(X)$ we have $A_X^{(m)}(\n)=A_X(\n)>0$ (since $U$ is smooth at the center of $\n$) while $\nu(\fa_m)=0$. 

Next, consider the functions $\phi,\phi_m:\Val_U^N\to\R\cup\{+\infty\}$ defined by 
$$
\phi(\n):=\n(f)-c\n(\fa_\bullet)+A_X(\n)
$$
and
$$
\phi_m(\n):=\n(f)-c m^{-1}\n(\fa_m)+A_X^{(m)}(\n)
$$
For each $m$ fixed, Proposition \ref{p:log a} and Theorem \ref{thm:extend} show that $\phi_m$ is lower semicontinuous, affine on small faces of $\Val_U^N$ and satisfies $\phi_m\ge\phi_m\circ r_\p^N$ for all $\p$ large enough. This shows that $\phi_m>0$ on $\DVal_U^N$ iff $\phi_m>0$ on $\Val_U^N$, \ie (ii)$\Longleftrightarrow$(iii). 

Further, each sublevel set $\{\phi_1\le a\}$ with $a\in\R$ is compact by Theorem \ref{thm:extend}, and Lemma \ref{lem:dini} thus yields (iv)$\Rightarrow$(iii), while the converse follows from $A_X\ge A_X^{(m)}$.

Since $\phi$ is lower semicontinuous and has compact sublevel sets, it achieves its infimum on $\Val_U^N$, which proves (iv)$\Longleftrightarrow$(v). Next, (v) trivially implies (vi), while the converse holds since $\phi$ is continuous on each dual complex $\D_\p^N$ and satisfies $\phi=\lim_\p\phi\circ r_\pi^N$ on $\Val_U^N$, again by Theorem \ref{thm:extend}. 

As to (vi)$\Longleftrightarrow$(vii), it holds because $\n(\I_N)=-1$ on $\Val_U^N$ by definition of the latter, while we have just as above $\nu(f)-c\n(\fa_\bullet)+A_X(\n)-\ep\n(\I_N)\ge A_X(\n)\ge 0$ on any $\nu\in\DVal_U$ centered outside $N$. 

To get (b) in Theorem \ref{t:mult-ideal} from (vii), note that the ideals 
$$
\left\{f\in\cO(U)\mid\n(f)\ge c\nu(\fa_\bullet)-A_X(\n)+\ep\n(\I_N)\text{ on }\DVal_U\right\}
$$
are independent of $0<\ep\ll 1$, by the Noetherian property of $\cO(U)$.
\end{proof}

\section{Numerically Cartier divisors}

\subsection{The group of numerical divisor classes} 
In this section, we provide an alternative and more concrete approach to the notion of numerically Cartier divisors introduced in \cite[Section 2]{BdFF}. 

As a matter of notation, we respectively denote by $\Car(X)$ and $Z^1(X)$ the groups of Cartier and Weil divisors of a normal variety $X$. We define the \emph{local class group of $X$} as 
$$
\Cll(X):=Z^1(X)/\Car(X). 
$$
By definition, $\Cll(X)$ is trivial iff $X$ is (locally) factorial. Since the usual divisor class group $\Cl(X)$  is defined as the quotient of $Z^1(X)$ by the subgroup of principal divisors, we have an exact sequence 
$$
0\to\Pic(X)\to\Cl(X)\to\Cll(X)\to 0.
$$
\begin{rmk} When $X$ only has an isolated singularity at $0\in X$, $\Cll(X)$ coincides with the divisor class group of the local ring $\cO_{X,0}$. 
\end{rmk}

\begin{defi}\label{defi:numcar} Let $X$ be a normal variety.
\begin{itemize}
\item[(i)] A Weil divisor $D\in Z^1(X)$ is \emph{numerically Cartier} if there exists a resolution of singularities $\mu:X'\to X$ (\ie a projective birational morphism with $X'$ smooth) and a $\mu$-numerically trivial Cartier divisor $D'$ on $X'$ such that $D=\mu_*D'$. 
\item[(ii)] We denote by $\NC(X)\subset Z^1(X)$ the subgroup of numerically Cartier divisors, and elements of $\NC(X)_\Q\subset Z^1(X')_\Q$ are called \emph{numerically $\Q$-Cartier}. 
\item[(iii)] The \emph{group of numerical divisor classes} of $X$ is defined as the quotient 
$$
\Cln(X):=Z^1(X)/\NC(X).
$$
\item[(iv)] We say that $X$ is \emph{numerically factorial} (resp. \emph{numerically $\Q$-factorial}) if $\Cln(X)=0$ (resp. $\Cln(X)_\Q=0$). 
\end{itemize}
\end{defi}
By definition, $\Cln(X)$ is a quotient of $\Cll(X)$, and it is in fact much smaller in general. Indeed, as we shall shortly see, $\Cln(X)$ is always finitely generated as an abelian group. 

In order to further analyze numerically Cartier divisors, we first show that it is enough to work with a \emph{fixed} resolution of singularities.

\begin{prop}\label{prop:numinde} Let $\mu:X'\to X$ be projective birational morphism. 
\begin{itemize}
\item[(i)] If $X'$ is factorial, every $D\in\NC(X)$ writes $D=\mu_*D'$ for a unique $\mu$-numerically trivial $D'\in\Car(X')$. 
\item[(ii)] If $X'$ is $\Q$-factorial, every $D\in\NC(X)_\Q$ is of the form $D=\mu_*D'$ for a unique $\mu$-numerically trivial $D'\in\Car(X')_\Q$.
\end{itemize}
In both cases, we set $\mu^*_\num D:=D'$ and call it the \emph{numerical pull-back} of $D$. 
\end{prop}

\begin{proof} The kernel of $\mu_*:Z^1(X')\to Z^1(X)$ is exactly the space of $\mu$-exceptional divisors. By the negativity lemma, there is no non-trivial divisor on $X'$ that is both $\mu$-numerically trivial and $\mu$-exceptional, which proves the uniqueness part in both cases. 

Now pick $D\in\NC(X)$. By definition, there exists a resolution $\mu'':X''\to X$ such that $D=\mu''_*D''$ for some $\mu''$-numerically trivial $D''\in\Car(X'')$. Since the pull-back of $D''$ to a higher resolution remains relatively numerically trivial, we may assume that $\mu''$ dominates $\mu$, \ie $\mu''=\mu\circ\rho$ for a birational morphism $\rho:X''\to X'$. Since $X'$ is factorial (resp. $\Q$-factorial), $D':=\rho_*D''$ belongs to $\Car(X')$ (resp. $\Car(X')_\Q$), and $D''-\rho^*D'$ is both $\rho$-exceptional and $\rho$-numerically trivial, hence trivial. By the projection formula, it follows that $D'$ is $\mu'$-numerically trivial and $D=\mu'_*D'$. 
\end{proof}

\begin{cor}\label{cor:exact} With the same assumption as in Proposition \ref{prop:numinde}, $\mu_*:Z^1(X')\to Z^1(X)$ induces:
\begin{itemize}
\item[(i)] an exact sequence of abelian groups
$$
0\to\Exc^1(\mu)\to N^1(X'/X)\to\Cln(X)\to 0
$$
if $X'$ is factorial, where $\Exc^1(\mu)$ is the (free abelian) group of $\mu$-exceptional divisors and $N^1(X'/X)$ is the group of $\mu$-numerical equivalence classes. 
\item[(ii)] an exact sequence of $\Q$-vector spaces
$$
0\to\Exc^1(\mu)_\Q\to N^1(X'/X)_\Q\to\Cln(X)_\Q\to 0
$$
if $X'$ is $\Q$-factorial.
\end{itemize}
In particular, $\Cln(X)$ is a finitely generated abelian group. 
\end{cor}

\begin{proof} The exact sequences in (i) and (ii) follow immediately from Proposition \ref{prop:numinde}. The last assertion is a consequence of the relative version of the theorem of the base \cite[p.334, Proposition 3]{Kle66}, which guarantees that $N^1(X'/X)$ is finitely generated.
\end{proof}

\begin{rmk}\label{rmk:special} As a special case of (ii) above, if $X'$ is $\Q$-factorial and $(E_i)$ denotes the $\mu$-exceptional prime divisors, then $X$ is numerically $\Q$-factorial iff for every $D'\in\Car(X')$ there exists $a_i\in\Q$ such that 
\begin{equation}\label{equ:dprime}
\left(D'+\sum_i a_i E_i\right)\cdot C=0
\end{equation}
holds for all curves $C\subset X'$ contained in a $\mu$-fiber.
\end{rmk}

\begin{eg}\label{eg:cone} If $X$ is an affine cone over a smooth projective polarized variety $(Y,L)$, then 
$$
\Cll(X)\simeq\Pic(Y)/\Z L
$$
and
$$
\Cln(X)\simeq\NS(Y)/\Z c_1(L).
$$ 
In particular, $X$ is numerically $\Q$-factorial iff $\rho(Y)=1$. 
\end{eg}

\begin{eg} Every surface is numerically $\Q$-factorial. This is directly related to the existence of Mumford's numerical pull-back. Indeed, let $\mu:X'\to X$ be a resolution of singularities, with exceptional divisor $\sum_i E_i$. Since the intersection matrix $(E_i\cdot E_j)$ is negative definite, for $D'\in\Car(X')$ we can find $a_i\in\Q$ such that (\ref{equ:dprime}) holds for each curve $C=E_j$.

Note that $\Cln(X)$ is however non trivial in general, even for a surface. For instance, it follows from Example \ref{eg:cone} that $\Cln(X)=\Z/2\Z$ for an $A_1$-singularity. 
\end{eg}

\begin{eg} If $X$ is log terminal in the sense of \cite{dFH09}, \ie if $(X,\D)$ is klt for some effective $\Q$-Weil divisor $\D$, it follows from the current knowledge in the Minimal Model Program that there exists a small projective birational morphism $\mu:X'\to X$ such that $X'$ is $\Q$-factorial. By (ii) of Corollary \ref{cor:exact}, we then have $N^1(X'/X)_\Q\simeq\Cln(X)_\Q$, which is thus trivial iff $\mu$ is an isomorphism. In other words, $X$ is numerically $\Q$-factorial iff $X$ is $\Q$-factorial. Since $X$ has rational singularities, the previous conclusion will also follow from Theorem \ref{thm:ratsing} below. 
\end{eg}

Let us now check that Definition \ref{defi:numcar} is indeed compatible with \cite[Definition 2.26, Remark 2.27]{BdFF}.\footnote{More precisely, numerically $\Q$-Cartier divisors in the present sense correspond to numerically Cartier divisors in the sense of \cite{BdFF}.}

\begin{prop}\label{prop:numcar} A Weil divisor $D$ on $X$ is numerically $\Q$-Cartier in the sense of Definition \ref{defi:numcar} iff 
\begin{equation}\label{equ:o}
\nu\left(\O_X(-mD)\right)=-\nu\left(\cO_X(mD)\right)+o(m)
\end{equation}
for all $\n\in \DVal_X$. 

For each projective birational morphism $\mu:X'\to X$ with $X'$ $\Q$-factorial, we then have
\begin{equation}\label{equ:dprime}
\lim_{m\to\infty}m^{-1}\nu\left(\O_X(-mD)\right)=\nu\left(\mu^*_\num D\right)
\end{equation}
for all $\nu\in \DVal_X$. In particular, the limit in the right-hand side is rational. 
\end{prop}

\begin{proof} In the terminology of \cite[Section 2]{BdFF}, (\ref{equ:o}) reads 
$$
\Env_X(-D)=-\Env_X(D),
$$ 
where $\Env_X(D)$ is the \emph{nef envelope} of $D$, \ie the $b$-divisor over $X$ characterized by
$$
\nu\left(\Env_X(D)\right)=\lim_{m\to\infty}m^{-1}\nu\left(\cO_X(mD)\right)
$$
for all $\nu\in \DVal_X$. Assume first that $D\in Z^1(X)$ is numerically $\Q$-Cartier. Let $\m:X'\to X$ be a projective birational morphism with $X'$ $\Q$-factorial and set $D':=\mu^*_\num D$. The Cartier $b$-divisor $\overline D'$ induced by pulling-back $D'$ is then relatively nef over $X$ and satisfies $\overline D'_X=D$, and hence 
$$
\overline D'\le\Env_X(D)
$$
by \cite[Proposition 2.12]{BdFF}. Since $-\overline D'$ is also relatively nef, we similarly get
$$
-\overline D'\le\Env_X(-D).
$$
Summing up these two inequalities and using the trivial inequality
$$
\Env_X(D)+\Env_X(-D)\le\Env_X(D-D)=0,
$$
we infer
$$
\Env_X(D)=\overline D',
$$
which proves (\ref{equ:o}) and (\ref{equ:dprime}). 

Conversely, assume that $D\in Z^1(X)$ satisfies $\Env_X(-D)=-\Env_X(D)$. By \cite[Lemma 2.10]{BdFF}, it follows that $D':=\Env_X(D)_{X'}\in\Car(X')_\R$ is $\mu$-numerically trivial. Since $\mu_*D'=D$ belongs to $Z^1(X)_\Q$ and $\mu_*$ is defined over $\Q$, the injectivity of $\mu_*$ on $\mu$-numerically trivial divisors implies that $D'$ is in fact a $\Q$-divisor, and hence that $D$ is numerically $\Q$-Cartier with $D'=\mu^*_\num D$. 
\end{proof} 

\begin{rmk} In particular, this result shows that the envelope $\Env_X(D)$ of a Weil divisor $D\in Z^1(X)$ such that $\Env_X(-D)=-\Env_X(D)$ is a $\Q$-Cartier $b$-divisor (the rationality of the coefficients being in particular not obvious from the definition). In fact, the whole point of the present point of view is to highlight the fact that the $\R$-vector space of $\R$-Weil divisors $D\in Z^1(X)_\R$ with $\Env_X(-D)=-\Env_X(D)$ is in fact defined over $\Q$. 
\end{rmk}

\subsection{The case of rational singularities}
In this section we prove:
\begin{thm} \label{thm:ratsing}  Let $X$ be a normal variety with at most rational singularities. Then 
$$
\NC(X)_\Q=\Car(X)_\Q,
$$  
\ie a Weil divisor is numerically $\Q$-Cartier iff it is $\Q$-Cartier. In particular, $X$ is numerically $\Q$-factorial iff $X$ is $\Q$-factorial. 
\end{thm}
The proof is inspired from that of \cite[Lemma 1.1]{Kaw88}, which states that the $\Q$-vector space $Z^1(X)_\Q/\Car(X)_\Q$ is finite dimensional when $X$ has rational singularities. We will need the following two results. 

\begin{lem}\label{lem:carfor}\cite[Proposition 1]{Sam61} A Weil divisor $D$ on $X$ is Cartier iff its restriction to the formal completion of $X$ at each (closed) point $x\in X$ is Cartier. 
\end{lem}

\begin{lem}\label{lem:numsing} If $Y$ is a (possibly reducible) projective complex variety, a line bundle $L$ on $Y$ is numerically trivial, \ie $L\cdot C=0$ for all curves $C\subset Y$, iff $c_1(L)=0$ in $H^2(Y,\Q)$. 
\end{lem}
\begin{proof} When $Y$ is non-singular, the result is well-known, and amounts to the Hodge conjecture for $1$-dimensional cycles (which follows from the $1$-codimensional case via the Hard Lefschetz theorem). However, we haven't been able to locate a reference in the literature in the general singular case; we are very grateful to Claire Voisin for having shown us the following argument. Let $\p:Y'\to Y$ be a resolution of singularities. Since $\pi^*L$ is also numerically trivial, we have $\pi^*c_1(L)=0$ in $H^2(Y',\Q)$, by the result in the smooth case. By \cite[Corollary 5.42]{PS}, this means that $c_1(L)\in W_1 H^2(Y,\Q)$, where $W_\bullet$ denotes the weight filtration of the mixed Hodge structure. The problem is thus to show that $W_1 H^2(Y,\Q)$ only meets the image of $\Pic(Y)$ at $0$. 

To see this, note that there exists a morphism $f:Y\to Z$ to a smooth projective variety $Z$ such that $L=f^*M$ for some line bundle $M$ on $Z$; indeed, this is true with $Z$ a projective space when $L$ is very ample, and writing $L$ as a difference of very ample line bundles gives the general case, with $Z$ a product of two projective spaces. 

Since $f^*:H^2(Z,\Q)\to H^2(Y,\Q)$ is a morphism of mixed Hodge structures, it is \emph{strict} with respect to weight filtrations, and we get
$$
c_1(L)\in f^*H^2(Z,\Q)\cap W_1 H^2(Y,\Q)=f^*\left(W_1 H^2(Z,\Q)\right),
$$
which is zero since $Z$ is smooth.
\end{proof}

\begin{proof}[Proof of Theoren \ref{thm:ratsing}] Let $\m: X' \to X$ be a resolution of singularities and let $D'\in\Car(X')$ be $\mu$-numerically trivial. Our goal is to show that $D:=\m_*D'$ is $\Q$-Cartier. By Lemma \ref{lem:carfor}, it is enough to show that every (closed) point $x\in X$ has an analytic neighborhood $U$ on which $D^\an$ is $\Q$-Cartier. 

The exponential exact sequence on the associated complex analytic variety $X'^\an$ yields an exact sequence
$$
R^1\m^\an_*\cO\to R^1\m^\an_*\cO^*\to R^2\m_*\Z\to R^2\m^\an_*\cO
$$
of sheaves on $X^\an$, where the two extreme term coincide by GAGA with the analytifications of $R^q\m_*\cO$ for $q=1,2$,   and hence vanish since $X$ has rational singularities. We thus have an isomorphism 
$$
\left(R^1\m^\an_*\cO^*\right)_x\simeq\left(R^2\mu^\an_*\Z\right)_x=H^2\left(\mu^{-1}(x),\Z\right),
$$
where the right-hand equality holds by properness of $\mu^\an$ (which again follows from GAGA). Since $D'$ has degree $0$ on each projective curve $C\subset\mu^{-1}(x)$, its image in $H^2\left(\mu^{-1}(x),\Q\right)$ is trivial by Lemma \ref{lem:numsing}. By the above isomorphism, the image of $D'^\an$ in $\left(R^1\m^\an_*\cO^*\right)_x\otimes\Q$ is also trivial, which means that $D'^\an$ is $\Q$-linearly equivalent to $0$ on $(\mu^\an)^{-1}(U)$ for a small enough analytic neighborhood $U$ of $x$. Since the morphism $(\mu^\an)^{-1}(U)\to U$ is a proper modification, it follows as desired that $D^\an$ is $\Q$-Cartier on $U$. 
\end{proof}

\subsection{Multiplier ideals in the numerically $\Q$-Gorenstein case}
\begin{defi} A normal variety $X$ is \emph{numerically $\Q$-Gorenstein} if $K_X$ is numerically $\Q$-Cartier.
\end{defi}
Given a resolution of singularities $\m:X'\to X$, Corollary \ref{cor:exact} shows that $X$ is numerically $\Q$-Gorenstein iff $K_{X'}$ is $\mu$-numerically equivalent to a $\mu$-exceptional $\Q$-divisor, which is then uniquely determined and denoted by $K^\num_{X'/X}$. In other words, we set
$$
K^\num_{X'/X}:=K_{X'}-\mu^*_\num K_X.
$$
By Proposition \ref{prop:numcar}, for each prime divisor $E\subset X'$ we then have
\begin{equation}\label{equ:anum}
A_X(\ord_E)=1+\ord_E\left(K^\num_{X'/X}\right).
\end{equation}

\begin{lem}\label{lem:numqgor} Assume that $X$ is numerically $\Q$-Gorenstein, and let $N\subset X$ be a normalizing subscheme. For each good resolution $\p$ of $N$, the log discrepancy function $A:\Val_X\to\R\cup\{+\infty\}$ is then affine on the faces of the dual complex $\D_\p^N$, and 
$$
A_X=\sup_\p A_X\circ r_\p^N
$$ 
on $\Val_X^N$, where $\p$ ranges over all good resolution of $N$. 
\end{lem}

\begin{proof} Write 
$$
K^\num_{X'/X}=K_{X'}-\p^*_\num K_X=\sum_i a_i E_i
$$
with $a_i\in\Q$ and $E_i$ $\p$-exceptional and prime. Modulo the canonical homeomorphism $\Val_{X'}\simeq\Val_X$, (\ref{equ:anum}) yields
$$
A_X(\n)=A_{X'}(\n)+\sum_i a_i\n(E_i)
$$
on $\Val_X$. Since $X'$ is smooth, $A_{X'}$ is affine on the faces of $\D_\p^N$ and satisfies $A_{X'}\ge A_{X'}\ge r_\p^N$. On the other hand, Proposition \ref{p:log a} shows that $\nu\mapsto\n(E_i)$ is also affine on the faces $\D_\p^N$, and satisfies $r_\p^N(\n)(E_i)=\n(E_i)$. The result follows. 
\end{proof}

The next result is Theorem \ref{t:3} from the introduction:
\begin{thm}\label{thm:numgor} Assume that $X$ is numerically $\Q$-Gorenstein, and let $\fa\subset\cO_X$ be an ideal sheaf. Let also $\m:X'\to X$ be a log resolution of $(X,\fa)$, so that $\m^{-1}\fa\cdot\cO_{X'}=\cO_{X'}(-D)$ with $D$ an effective Cartier divisor. For each exponent $c>0$ we then have
$$
\J(X,\fa^c) =\m_*\cO_{X'}\left(\lceil K_{X'/X}^\num-c D\rceil\right).
$$
\end{thm}
\begin{proof} Let $N$ be a normalizing subscheme containing the zero locus of $\fa$, and pick a good resolution $\p$ of $N$ factoring as $\p=\mu\circ\r$. Using Lemma \ref{lem:numqgor} and arguing as in the proof of Theorem \ref{t:mult-ideal}, we easily get
$$
\J(X,\fa^c)=\p_*\cO_{X_\p}\left(\lceil K_{X_\p/X}^\num-c D\rceil\right). 
$$
On the other hand, we have
$$
K_{X_\p/X}^\num=\r^*K^\num_{X'/X}+K_{X_\p/X'},
$$ 
since both sides of the equality are $\p$-exceptional and $\p$-numerically equivalent to $K_{X_\p}$. Since $K_{X_\p/X'}$ is effective and $\p$-exceptional, we obtain as desired
$$
\p_*\cO_{X_\p}\left(\lceil K_{X_\p/X}^\num-c D\rceil\right)=\m_*\cO_{X'}\left(\lceil K_{X'/X}^\num-c D\rceil\right).
$$
\end{proof}

\begin{cor} Assume that $X$ is numerically $\Q$-Gorenstein. Then $X$ has log terminal singularities (in the usual sense, \ie with $K_X$ is $\Q$-Cartier) iff $A_X>0$ on $\DVal_X$.
\end{cor}
\begin{proof} By Theorem \ref{thm:numgor}, we have $A_X>0$ on $\DVal_X$ iff $\cJ(X,\cO_X)=\cO_X$, which is the case iff there exists an effective $\Q$-Weil divisor $\D$ such that the pair $(X,\D)$ is klt \cite{dFH09} (see also \cite{BdFF}). But this implies that $X$ has rational singularities, and Theorem \ref{thm:ratsing} thus shows that $K_X$ is $\Q$-Cartier. Since $(X,\D)$ is klt, so is $(X,0)$, which means that $X$ is log terminal in the classical sense.
\end{proof}

\vspace{0.2cm}
\setlength{\parindent}{0in}
\def\scshape{}

\end{document}